\documentclass[english,11pt]{amsart}

\usepackage{babel}
\usepackage{amstext}
\usepackage{amsmath}
\usepackage{amsfonts}
\usepackage{latexsym}
\usepackage{ifthen}
\usepackage{amsthm}
\usepackage{amssymb}
\usepackage[all]{xy}
\usepackage{textcomp}

\newcommand{\N}{\mathcal{N}_1}

\newcommand{\NE}{\operatorname{NE}}
\newcommand{\Exc}{\operatorname{Exc}}
\newcommand{\Locus}{\operatorname{Locus}}
\newcommand{\codim}{\operatorname{codim}}
\newcommand{\reg}{\operatorname{reg}}

\newtheorem{thm}{Theorem}[section]
\newtheorem{lem}[thm]{Lemma}

\newtheorem{pro}[thm]{Proposition}

\theoremstyle{definition}

\numberwithin{equation}{section}

\begin{document}
\title{A note on the Picard number of singular Fano $3$-folds}
\author{Gloria Della Noce}
\address{Dipartimento di Matematica \\ Universit\`a  di Pavia \\ via Ferrata, 1 27100 Pavia - Italy}
\email{gloria.dellanoce@unipv.it}

\thanks{
This work has been partially supported by PRIN 2009 \mbox{\textquotedblleft \textit{Moduli, strutture geometriche e loro applicazioni}\textquotedblright}.}


\begin{abstract}

\bigskip

\noindent Using a construction due to C. Casagrande and further developed by the author in \cite{gloria}, we prove that the Picard number of a non-smooth Fano $3$-fold with isolated factorial canonical singularities, is at most $6$.

\end{abstract}

\maketitle

\bibliographystyle{amsalpha}

\section*{Introduction}

Let $X$ be a Fano $3$-fold. If $X$ is smooth, we know from the classifiction results in \cite{mm}, that its Picard number $\rho_X$ is at most $10$. Moreover, if $\rho_X \geq 6$, then $X$ is isomorphic to a product $S \times \mathbb{P}^1$, where $S$ is a smooth Del Pezzo surface.

If $X$ is singular, bounds for $\rho_X$ are known only in particular cases. If $X$ is toric and has canonical singularities, then $\rho_X \leq 5$ (\cite{bat3} and \cite{ww}). If $X$ has Gorenstein terminal singularities, then $\rho_X \leq 10$, because $X$ has a smoothing which preserves $\rho_X$ (see \cite[Thorem 11]{nam} and \cite[Theorem 1.4]{jr}). If, instead, $X$ has Gorenstein canonical singularities, it does not admit, in general, a smooth deformation (see \cite[Example 1.4]{prok} for an example). In this setting, the following holds.

\begin{thm}\cite[Theorem 1.3]{gloria}\label{dim3} Let $X$ be a $3$-dimensional $\mathbb{Q}$-factorial Gorenstein Fano variety with isolated canonical singularities. Then $\rho_X \leq 10$.
\end{thm}

The proof of this theorem uses a construction introduced by C. Casagrande in \cite{cas12}, and relies on the result of \cite{bchm} that Fano varieties are \emph{Mori dream spaces} (see \cite{hk} for the definition).

In this paper, using the same construction, we show that the bound given by Theorem \ref{dim3} can be improved if $X$ is actually singular and its singularities are also factorial. Our result is the following.

\begin{thm}\label{3fac}
Let $X$ be a non-smooth factorial Fano $3$-fold with isolated canonical singularities. Then $\rho_X \leq 6$.
\end{thm}

In the first section of this paper, we recall some preliminary results from \cite{gloria}; the second section contains the proof of Theorem \ref{3fac} and an observation concerning the case $\rho_X=6$.
\vspace{.3cm}

\noindent\textbf{Notation and terminology}

\medskip
We work over the field of complex number.

Let $X$ be a normal variety. We call $X$ \emph{Fano} if $-K_X$ has a multiple which is an ample Cartier divisor. We denote by $X_{reg}$ the non-singular locus of $X$. We say that $X$ is \emph{$\mathbb{Q}$-factorial} if every Weil divisor is $\mathbb{Q}$-Cartier, \emph{i.e.} it admits a multiple which is Cartier. We call $X$ \emph{factorial} if all its local rings are UFD; by \cite[II, Proposition 6.11]{H}, this implies that every Weil divisor of $X$ is Cartier. We refer the reader to \cite{km} for the definition and properties of terminal and canonical singularities. If $X$ has canonical singularities, it is called \emph{Gorenstein} if its canonical divisor $K_X$ is a Cartier divisor.

We denote with $\N(X)$ (resp. $\mathcal{N}^1(X)$) the vector space of one-cycles (resp. $\mathbb{Q}$-Cartier divisors) with real coefficients, modulo the relation of numerical equivalence. The dimension of these two real vector spaces is, by definition, the \emph{Picard number of $X$}, and is denoted by $\rho_X$. We denote by $[C]$ (resp. $[D]$) the numerical equivalence class of a one-cycle (resp. a $\mathbb{Q}$-Cartier divisor). 

Given $[D]\in \mathcal{N}^1(X)$, we set $D^{\perp}:=\{\gamma \in \N(X) | D \cdot \gamma =0\}$, where $\cdot$ denotes the intersection product. We define $\NE(X) \subset \N(X)$ as the convex cone generated by classes of effective curves and $\overline{\NE}(X)$ is its closure. An \textit{extremal ray} $R$ of $X$ is a one-dimensional face of $\overline{\NE}(X)$. We denote by $\Locus(R)$ the subset of $X$ given by the union of curves whose class belongs to $R$.

A \textit{contraction} of $X$ is a projective surjective morphism with connected fibers $\varphi:X \to Y$ onto a projective normal variety $Y$. It induces a linear map $\varphi_*:\N(X) \to \N(Y)$ given by the push-forward of one-cycles. We set $\NE(\varphi):=\overline{\NE}(X) \cap \ker(\varphi_*)$. We say that $\varphi$ is $K_X$-negative if $K_X \cdot \gamma <0$ for every $\gamma \in \NE(\varphi)$.

The \textit{exceptional locus} of $\varphi$ is the locus where $\varphi$ is not an isomorphism; we denote it by $\Exc(\varphi)$. We say that $\varphi$ is \textit{of fiber type} if $\dim(X) > \dim(Y)$, otherwise $\varphi$ is birational. We say that $\varphi$ is \textit{elementary} if $\dim(\ker(\varphi_*))=1$. In this case $\NE(\varphi)$ is an extremal ray of $\overline{NE}(X)$; we say that $\varphi$ (or $\NE(\varphi)$) is \textit{divisorial} if $\Exc(\varphi)$ is a prime divisor of $X$ and it is \textit{small} if its codimension is greater than $1$.

An elementary contraction from a $3$-fold $X$ is called \emph{of type $(2,1)$} if $\varphi$ is $K_X$-negative, birational, $\dim(\Exc(\varphi))=2$ and $\dim(\varphi(\Exc(\varphi)))=1$.

If $D \subset X$ is a Weil divisor and $i:D \to X$ is the inclusion map, we set $\N(D,X):=i_*\N(D) \subseteq \N(X)$.

\section{Preliminaries}

In the following statement, we collect some results from \cite{gloria}. For the reader's convenience, we recall here the main steps of their proof. We refer the reader to \cite[Theorem 2.2]{gloria} for the properties of contractions of type $(2,1)$ defined on mildly singular $3$-folds. 

\begin{lem}\cite[Theorem 1.2 and its proof - Remark 5.2]{gloria}\label{lemDN}
Let $X$ be a $\mathbb{Q}$-factorial Gorenstein Fano $3$-fold with isolated canonical singularities. Suppose $\rho_X \geq 6$. Then there exist morphisms
\[
\psi:X \to \mathbb{P}^1 \ \ \ \textmd{and} \ \ \ \xi: X\to S,
\]
where $S$ is a normal surface with $\rho_S=\rho_X -1$, and the morphism
\[
\pi:=(\xi,\psi):X \to S \times \mathbb{P}^1
\]
is finite.

Moreover there exist extremal rays $R_0,\ldots,R_m$ ($m \geq 3$) in $\NE(X)$ such that:
\begin{itemize}
\item each $R_i$ is of type $(2,1)$;
\item $\NE(\psi)=R_0 + \cdots + R_m$;
\item for $i=0,\ldots, m$, set $E_i=\Locus R_i$ and $Q=\NE(\xi)$. Then
\[
\psi(E_i)=\mathbb{P}^1, \ \ \  \N(E_i,X)=\mathbb{R}R_i \oplus \mathbb{R}Q\ \ \ \textmd{and} \ \ \ Q \subseteq \bigcap_{i=0}^m E_i^{\perp};
\]
\item $\psi$ factors as $X \stackrel{\sigma}{\to} \tilde{X} \to \mathbb{P}^1$, where $\sigma$ is birational, $\tilde{X}$ is a Fano $3$-fold with canonical isolated singularities, $\NE(\sigma)=R_1 + \cdots + R_s$, with $m \geq s\in\{\rho_X-2,\rho_X-3\}$ and $\sigma(E_1),\ldots,\sigma(E_s)\subset \tilde{X}$ are pairwise disjoint.
\end{itemize}
\end{lem}

\begin{proof}
By \cite[Remark 5.2]{gloria}, the assumption $\rho_X \geq 6$ implies that all the assumptions of \cite[Theorem 1.2]{gloria} are satisfied, from which the existence of the finite morphism $\pi$. The properties of its projections $\psi$ and $\xi$ follow by their construction, that we briefly recall. All the details can be found in the proof of \cite[Theorem 1.2]{gloria}.

By \cite[Proposition 3.5]{gloria}, there exists an extremal ray $R_0 \subset \NE(X)$ of type $(2,1)$. Set $E_0=\Locus(R_0)$; we have $\dim\N(E_0,X)=2$. As in \cite[Lemma 3.1]{gloria}, we may find a Mori program
\begin{equation}\label{mp}
X=X_0 \stackrel{\sigma_0}{\dashrightarrow} X_1 \dashrightarrow \cdots \dashrightarrow X_{k-1} \stackrel{\sigma_{k-1}}{\dashrightarrow}X_k \xrightarrow{\varphi} Y
\end{equation}
where $X_1,\ldots,X_k$ are $\mathbb{Q}$-factorial $3$-folds with canonical singularities and, for each $i=0,\ldots,k-1$, there esists a $K_{X_i}$-negative extremal ray $Q_i \subset \NE(X_i)$ such that $\sigma_i$ is either its contraction, if $Q_i$ is divisorial, or its flip, if it is small. Moreover, if $(E_0)_i \subset X_i$ is the transform of $E_0$ and $(E_0)_0:=E_0$, then $(E_0)_i \cdot Q_i>0$. Finally, $\varphi$ is a fiber type contraction to a $\mathbb{Q}$-factorial normal variety $Y$.

Let us set
\[
\{i_1,\ldots,i_s\}:=\{i \in \{0,\ldots,k-1\} | \codim \N(D_{i+1},X_{i+1})=\codim \N(D_i,X_i)-1\}.
\]
Then, by \cite[Lemma 3.3]{gloria}, $s \in \{\rho_X-2,\rho_X-3\}$ (in particular $s \geq 3$); moreover, for every $j \in \{1,\ldots,s\}$, $Q_{i_j}$ is a divisorial ray, $\sigma_{i_j}$ is a birational contraction of type $(2,1)$ and, if $E_j \subset X$ is the transform of the exceptional divisors of the contraction $\sigma_{i_j}$ as above, then $E_1,\ldots,E_s$ are pairwise disjoint.

Since $s \geq 3$, \cite[Proposition 3.5]{gloria} assures that, for each $j=1,\ldots,s$, there exists an extremal ray $R_j \subset \NE(X)$ of type $(2,1)$ such that $E_j=\Locus(R_j)$. The divisor $-K_X +E_1+\cdots+E_s$ comes out to be nef, and its associated contraction $\sigma:X \to \tilde{X}$ verifies
\[
\ker(\sigma_*)=\mathbb{R}R_1+\cdots+\mathbb{R}R_s \ \ \ \textmd{and} \ \ \ \Exc(\sigma)=E_1 \cup \cdots \cup E_s.
\]
It is thus possible to look at $\sigma$ a part of a Mori program as in \eqref{mp}, and to find a fiber type contraction $\varphi:\tilde{X}\to Y$ giving rise to a morphism $\psi:= \varphi \circ \sigma: X\to Y$ as in the statement. In particular, we have $\NE(\psi)=R_0 + \cdots + R_m$, where $m \geq s$ and $R_{s+1},\ldots,R_m$ are extremal rays of type $(2,1)$. We notice that, since $\dim(X)=3$, we have $Y \cong \mathbb{P}^1$ by \cite[Remark 4.2]{gloria}.

The second projection $\xi$ arises as the contraction associated to a certain nef divisor defined as a combination of the prime divisors $E_0, \ldots, E_m$ constructed above (recall that $E_i=\Locus R_i$ for $i=0,\ldots,m$). It is an elementary contraction and the one-dimensional subspace generated by $\NE(\xi)$ belongs to $\N(E_i,X)$ for every $i=0,\ldots,m$.
\end{proof}

\section{Theorem \ref{3fac}}

\begin{proof}[Proof of Theorem \ref{3fac}]
Let us prove that, if $\rho_X \geq 7$, then the morphism $\pi:X \to S \times \mathbb{P}^1$ given by Lemma \ref{lemDN} is an isomorphism. This will give a contradiction with our assumptions on the singularities of $X$, since $S \times \mathbb{P}^1$ is smooth or has one-dimensional singular locus.

We are in the setting of Lemma \ref{lemDN}; let us keep its notations. By \cite[Corollary 1.9 and Theorem 4.1(2)]{AW}, the general fiber of $\xi$ is a smooth rational curve, and the other fibers have at most two irreducible components (that might coincide) whose whose reduced structures are isomorphic to $\mathbb{P}^1$.

Our assumptions imply that $S$ is factorial: if $C \subset S$ is a Weil divisor, its counterimage $D:=\xi^{-1}(C) \subset X$ is a Cartier divisor, because $X$ is factorial. Moreover $D \cdot Q=0$ (where $Q=\NE(\xi)$), because $D$ is disjoint from the general fiber of $\xi$. Then $D=\xi^*(C')$ for a certain Cartier divisor $C'$ on $S$. But then $C = C'$ is Cartier. 

Fix $i=0,\ldots,m$; let $\varphi_i:X \to Y_i$ be the contraction of $R_i$ and set $G_i:=\varphi_i(E_i) \subseteq Y_i$, $T_i:=\xi(E_i)\subseteq S$:
\[
\xymatrix{
& E_i \ar[dl]_{\varphi_{i|E_i}} \ar[dr]^{\xi_{|E_i}}\\
G_i && T_i.
}
\]
Notice that $T_i \subset S$ is a curve. Indeed, by Lemma \ref{lemDN}, $E_i \cdot Q=0$, which implies that $T_i \subset S$ is a (Cartier) divisor and $E_i=\xi^*(T_i)$.

Let $f_i$ be the general fiber of $\varphi_i$. Since $f_i$ is a smooth rational curve which dominates $T_i$, $T_i$ is a (possibly singular) rational curve. The same conclusion holds for $G_i$, which is dominated by any smooth curve contained in a fiber of $\xi$ over $T_i$.

We have
\[
-1=E_i \cdot f_i = \xi^*(T_i) \cdot f_i = T_i^2 \cdot \deg(\xi_{|f_i}),
\]
from which $-T_i^2=\deg(\xi_{|f_i})=1$. Then the general fiber $g$ of $\xi$ over $T_i$ is a smooth rational curve. Indeed, $g$ has no embedded points, and if, by contradiction, the $1$-cycle associated to $g$ is of the type $C_1 + C_2$, then $g$ would intersect $f_i$ in at least two (distinct or coincident) points. This is impossible because $g$ is general and $\deg(\xi_{|f_i})=1$.

Then $E_i$ is smooth along the general fibers of both $\varphi_i$ and $\xi$; we deduce that $E_i$ is smooth in codimension one. Moreover $E_i$ is a Cohen-Macaulay variety, because $X$ is factorial. Then, by Serre's criterion, $E_i$ is normal. Then the finite morphism $(\xi_{|E_i},\varphi_{i|E_i}):E_i \to T_i \times G_i$, which has degree one, factors through the normalization of the target: there is a commutative diagram
\[
\xymatrix{
E_i \ar[dr] \ar[r]^{\hspace{-.3cm}\tau} & \mathbb{P}^1 \times \mathbb{P}^1 \ar[d]^{\nu} \\
& T_i \times G_i. \\
}
\]
Since $\tau$ is finite of degree one, by Zariski Main Theorem, it is an isomorphism. Thus $E_i \cong \mathbb{P}^1 \times \mathbb{P}^1$, and $\xi_{|E_i}:E_i \to T_i \cong \mathbb{P}^1$ and $\varphi_{i|E_i}:E_i \to G_i \cong \mathbb{P}^1$ are the projections. In particular, since both $E_i$ and $T_i$ are Cartier divisors, they are contained in the smooth loci of, respectively, $X$ and $S$.

We have
\[
(K_X - \xi^*(K_S))\cdot f_i = (K_{E_i} - \xi_{|E_i}^*(K_{T_i})) \cdot f_i = (\varphi_{i|E_i}^*(K_{G_i}))\cdot f_i =  0.
\]

Let $F$ be a general fiber of $\psi:X \to \mathbb{P}^1$. Then $F$ is a smooth Del Pezzo surface and, by Lemma \ref{lemDN}, $\N(F) \subseteq \sum \mathbb{R}[f_i]$; thus $K_X - \xi^*(K_S)$ is numerically trivial in $F$. Moreover $\zeta:=\xi_{|F}:F \to S$ is a finite morphism of degree $d:=\deg(\pi)$ and
\begin{equation}\label{F-S}
K_F=(K_X)_{|F}=(\xi^*(K_S))_{|F}=\zeta^*K_S;
\end{equation}
in particular $\zeta$ is unramified in the open subset $\xi^{-1}(S_{\reg})$, which contains $E_i \cap F$ for every $i=0,\ldots,m$.

Set $\tilde{F}:=\sigma(F) \subset \tilde{X}$, where $\sigma:X \to \tilde{X}$ is the birational contraction given by Lemma \ref{lemDN}; then $\tilde{F}$ is again a smooth Del Pezzo surface and $\sigma_{|F}:F \to \tilde{F}$ is a contraction. For every $i=1,\ldots,s$, the intersection $E_i \cap F$ is the union of $d$ disjoint curves numerically equivalent to $f_i$; in particular $\sigma_{|F}$ realizes $F$ as the blow-up of $\tilde{F}$ along $s\cdot d$ distinct points (where $s=\rho_X - \rho_{\tilde{X}}$). Then, recalling that $s \geq \rho_X-3$ and $\rho_X \geq 7$, we get
\[
9 \geq \rho_F=\rho_{\tilde{F}} + s \cdot d \geq 1 + 4 d,
\]
and then $d \leq 2$. Moreover, if $d=2$, then $\rho_F=9$ and, by \ref{F-S},
\[
1=K_{F}^2=\zeta^*(K_S) \cdot K_F =2 (K_S)^2,
\]
which is impossible because $S$ is factorial and thus $K_S^2$ is integral. Hence $d=\deg(\zeta)=\deg(\pi)=1$ and the statement is proved.
\end{proof}

The case $\rho_X=6$ is more complicated to analyze. Indeed, though Lemma \ref{lemDN} still holds in that case, we are not able to conclude that $\pi$ is an isomorphism and that, as a consequence, $X$ is smooth.

\begin{pro}
Let $X$ be a factorial Fano $3$-fold with isolated canonical singularities and with $\rho_X=6$. If $X$ is not smooth, there exists a finite morphism of degree $2$
\[
 \pi:X \to S \times \mathbb{P}^1,
\]
where $S$ is a singular Del Pezzo surface with factorial canonical singularities, $\rho_S=5$, $(K_S)^2=1$. Moreover the ramification locus of $\pi$ contains a surface $R$ which dominates $S$. 
\end{pro}

\begin{proof}
 We argue as in the proof of Theorem \ref{3fac} and we use the same notations. Since $X$ is not smooth, the degree of $\pi$ must be $2$. Exactly as in the above case, we have
\begin{equation}\label{miao}
 K_F=(K_X)_{|F}=(\xi^*(K_S))_{|F}=(\xi^*K_S)_{|F} =\zeta^*K_S,
\end{equation}
and
\begin{equation}\label{pari}
 \rho_F= 10 - (K_F)^2 = 10 - 2(K_S)^2,
\end{equation}
so that $\rho_F$ needs to be even. Since $\rho_X=6$, we have $s \in \{3,4\}$, and then
\[
9 \geq \rho_F=\rho_{\tilde{F}} + 2s.
\]
Thus the only possibility is that $\rho_{\tilde{F}}=2$ and $\rho_F=8$. By \eqref{pari}, we get $(K_S)^2=1$.

Let us call $R$ the ramification divisor (possibly trivial) of $\pi$. Let $C$ be the general fiber of $\xi$. Then $C \cong \mathbb{P}^1$ and $\psi_{|C}:\mathbb{P}^1 \to \mathbb{P}^1$ is finite of degree $2$. By Hurwitz's formula we have $R \cdot C=2$, and hence $R$ is not trivial and it dominates $S$.
\end{proof}

\vspace{.3cm}

\noindent\textit{Acknowledgments. }This paper is part of my PhD thesis; I am deeply grateful to my advisor Cinzia Casagrande for her constant guidance.

\medskip

\providecommand{\bysame}{\leavevmode\hbox to3em{\hrulefill}\thinspace}
\providecommand{\MR}{\relax\ifhmode\unskip\space\fi MR }
\providecommand{\MRhref}[2]{%
  \href{http://www.ams.org/mathscinet-getitem?mr=#1}{#2}
}
\providecommand{\href}[2]{#2}


\begin{thebibliography}{BCHM10}

\bibitem[AW97]{AW}
M.~Andreatta and J.~A. Wi\'{s}niewski, \emph{A view on contractions of higher
  dimensional varieties}, Algebraic Geometry-Santa Cruz 1995, Proc. Symp. Pure
  Math., vol. 62, 1997, pp.~153--183.

\bibitem[Bat82]{bat3}
V.~V. Batyrev, \emph{Toric {F}ano threefolds ({E}nglish traslation)}, Math.
  USSR-Izv. \textbf{19} (1982), 13--25.

\bibitem[BCHM10]{bchm}
C.~Birkar, P.~Cascini, C.~D. Hacon, and J.~Mc{K}ernan, \emph{{E}xistence of
  minimal models for varieties of log general type}, {J}. Amer. Math. Soc.
  \textbf{23} (2010), 405--468.

\bibitem[Cas12]{cas12}
C.~Casagrande, \emph{On the {P}icard number of divisors in {F}ano manifolds},
  {A}nn. {S}ci. {\'E}c. {N}orm. {S}up{\'e}r. \textbf{45} (2012), 363--403.

\bibitem[DN12]{gloria}
G.~Della~Noce, \emph{On the {P}icard number of singular {F}ano varieties},
  preprint arXiv:1202.1154 (2012), to appear in {I}nternat. Math. Res. Notices.

\bibitem[Har77]{H}
R.~Hartshorne, \emph{Algebraic {G}eometry}, Springer, 1977.

\bibitem[HK00]{hk}
Y.~Hu and S.~Keel, \emph{Mori dream spaces and {GIT}}, Michigan Math. J.
  \textbf{48} (2000), 331--348.

\bibitem[JR11]{jr}
P.~Jahnke and I.~Radloff, \emph{Terminal {F}ano threefolds and their
  smoothing}, Math. Zeitschrift \textbf{269} (2011), 1129--1136.

\bibitem[KM98]{km}
J.~Koll\'{a}r and S.~Mori, \emph{Birational {G}eometry of {A}lgebraic
  {V}arieties}, Cambridge Univ. Press, 1998.

\bibitem[MM81]{mm}
S.~Mori and {S}. Mukai, \emph{Classification on {F}ano $3$-folds with $b_2 \geq
  2$}, Manuscr. {M}ah. \textbf{36} (1981), 147--162, \emph{Erratum}:
  \textbf{110} (2003), 407.

\bibitem[Nam97]{nam}
Y.~Namikawa, \emph{Smoothing {F}ano $3$-folds}, J. Alg. Geom. \textbf{6}
  (1997), 307--324.

\bibitem[Pro05]{prok}
Y.~G. Prokhorov, \emph{On the degree of {F}ano threefolds with canonical
  {G}orenstein singularities ({E}nglish translation)}, Sb. Math. \textbf{196}
  (2005), 77--114.

\bibitem[WW82]{ww}
K.~Watanabe and M.~Watanabe, \emph{The classification of {F}ano $3$-folds with
  torus embeddings}, Tokio J. Math. \textbf{5} (1982), 37--48.

\end{thebibliography}

\end{document}